\documentclass{article}
\usepackage[utf8]{inputenc}
\usepackage{fullpage,amsfonts,amsmath,amsthm,amssymb,mathtools,mathrsfs,graphicx}
\usepackage{hyperref}

\hypersetup{
	colorlinks,
	citecolor=blue,
	filecolor=blue,
	linkcolor=blue,
	urlcolor=blue,
	linktocpage
} 

\usepackage{color}
\newcommand{\cp}{\mathcal{P}}

\newcommand{\N}{\mathcal{N}}
\newcommand{\F}{\mathcal{F}}
\newcommand{\PM}{\mathrm{pm}}
\newcommand{\eps}{\varepsilon}
\newcommand{\sub}{\subseteq}
\newcommand{\f}[2]{\frac{#1}{#2}}
\newcommand{\Z}{\mathbb{Z}}
\newcommand{\E}{\mathbb{E}}
\newcommand{\B}{\mathcal{B}}
\newcommand{\C}{\mathcal{C}}
\newcommand\lrpar[1]{\left(#1\right)}
\newcommand\floor[1]{\left\lfloor #1 \right \rfloor}
\newcommand\abs[1]{\left|#1\right|} 
\newcommand{\vc}{v^{(c)}}
\newcommand{\TV}{\mathrm{d}_{\mathrm{TV}}}
\newcommand{\Po}{\mathrm{Po}}
\newcommand{\gam}{\gamma}

\newtheorem{thm}{Theorem}[section]
\newtheorem{lemma}[thm]{Lemma}

\newtheorem{prop}[thm]{Proposition}
\newtheorem{claim}[thm]{Claim}
\newtheorem{conj}[thm]{Conjecture}

\usepackage{tikz}
\usepackage{extarrows}
\pgfdeclarelayer{nodelayer}
\pgfdeclarelayer{edgelayer}
\pgfsetlayers{edgelayer,nodelayer,main}
\tikzstyle{black}=[fill=black, draw=black, shape=circle, scale=0.3]
\tikzstyle{none}=[]

\newtheoremstyle{LemmaNum}
{\topsep}{\topsep}
{\itshape}
{}
{\bfseries} 
{.}
{ }
{\thmname{#1}\thmnote{ \bfseries #3}}
\theoremstyle{LemmaNum}

\newtheorem{propn}{Proposition}

\title{Counting Deranged Matchings}
\author{Sam Spiro\footnote{Dept.\ of Mathematics, Rutgers University {\tt sas703@scarletmail.rutgers.edu}. This material is based upon work supported by the National Science Foundation Mathematical Sciences Postdoctoral Research Fellowship under Grant No. DMS-2202730.} \and Erlang Surya \footnote{Dept.\ of Mathematics, University of California, San Diego {\tt esurya@ucsd.edu}. This material is based upon work supported by the National Science Foundation Grant No. DMS-2225631.}}
\begin{document}
	
	\maketitle
	
	\begin{abstract}
		Let $\mathrm{pm}(G)$ denote the number of perfect matchings of a graph $G$, and let $K_{r\times 2n/r}$ denote the complete $r$-partite graph where each part has size $2n/r$.   Johnson, Kayll, and Palmer conjectured that for any perfect matching $M$ of $K_{r\times 2n/r}$, we have for $2n$ divisible by $r$
		\[\frac{\mathrm{pm}(K_{r\times 2n/r}-M)}{\mathrm{pm}(K_{r\times 2n/r})}\sim e^{-r/(2r-2)}.\]
		This conjecture can be viewed as a common generalization of counting the number of derangements on $n$ letters, and of counting the number of deranged matchings of $K_{2n}$.  We prove this conjecture.  In fact, we prove the stronger result that if $R$ is a uniformly random perfect matching of $K_{r\times 2n/r}$, then the number of edges that $R$ has in common with $M$ converges to a Poisson distribution with parameter $\frac{r}{2r-2}$.
	\end{abstract}
	
	\section{Introduction}
	
	A \textit{derangement} is a permutation $\pi$ with $\pi_i\ne i$ for any $i$, and the number of derangements of size $n$ is denoted by $d_n$.  Derangements are one of the most well studied objects in combinatorics, whose origins date back to 1708 with work of de Montmort~\cite{de1713essay}.

	A common interpretation of derangements comes from the hat check problem: if $n$ people each have a hat and these hats are randomly distributed back to each person, what is the probability that no person receives their own hat?  This probability is easily seen to be equal to $d_n/n!$, and it is by now a standard exercise in the principle of inclusion-exclusion to show the surprising fact that
	\begin{equation}
		\f{d_n}{n!}\sim e^{-1}.\label{eq:deranged}\end{equation}
	We refer the interested reader to Stanley~\cite{stanley2011enumerative} for further results and problems related to derangements.

	The asymptotic formula \eqref{eq:deranged} can be interpreted in terms of graph theory.  To this end, let $\PM(G)$ denote the number of perfect matchings of a graph $G$.  It is not difficult to see $\PM(K_{n,n})=n!$ and $\PM(K_{n,n}-M)=d_n$ for any perfect matching $M$ of $K_{n,n}$.  Thus \eqref{eq:deranged} is equivalent to saying
	\begin{equation}\f{\PM(K_{n,n}-M)}{\PM(K_{n,n})}\sim e^{-1}.\label{eq:graphderange}\end{equation}
	
	Another question that can be phrased in the language above is the kindergarten problem: if $2n$ kindergarteners pair up for a field trip at the start of the day and then are randomly paired up again at the end of the day, what is the probability that no pair of kindergartners are matched at both the start and end of the day?  This problem was motivated by a question from the United States Tennis Association about the tournament draw for the 1996 U.S.\ Open, and it was reformulated in terms of kindergarteners by Kayll.
	
	Brawner~\cite{brawner2000dinner} conjectured that the probability for the kindergarten problem tends to $e^{-1/2}$, and this was proven by Margolius~\cite{margolius2001avoiding}.  By viewing the kindergartners as the vertices of $K_{2n}$, this result is equivalent to saying
	\begin{equation}\f{\PM(K_{2n}-M)}{\PM(K_{2n})}\sim e^{-1/2}.\label{eq:derangedMatchings}\end{equation}
	Noticing the similarities between \eqref{eq:graphderange} and \eqref{eq:derangedMatchings}, Johnston, Kayll, and Palmer~\cite{johnston2022deranged} made the following conjecture concerning perfect matchings of $K_{r\times 2n/r}$, the complete $r$-partite graph with parts of sizes $2n/r$.
	\begin{conj}\label{conj:main}
		If $r=r(n)\ge 2$ is integer valued and divides $2n$, and if $M_n$ is a sequence of perfect matchings in $K_{r\times 2n/r}$, then
		\[\f{\PM(K_{r\times 2n/r}-M_n)}{\PM(K_{r\times 2n/r})}\sim e^{-r/(2r-2)},\]
		as $n$ tends to infinity. 
	\end{conj}
	As an aside, we note that while Conjecture~\ref{conj:main} first appeared formally in \cite{johnston2022deranged} earlier this year, the conjecture has been around for some time now.  Indeed, Kayll first wrote about this in his Fullbright proposal in 2014, and we first heard about Conjecture~\ref{conj:main} in 2017 when Palmer spoke about it at a conference.

	Some special cases of Conjecture~\ref{conj:main} were proven in \cite{johnston2022deranged}.  In particular, the cases $r=3$ and $r$ linear in $n$ were proven using a mixture of inclusion-exclusion arguments and Tannery's theorem, and the case when $r=\Omega(n^\delta)$ for any fixed $\delta>0$ was solved using a powerful result of McLeod~\cite{mcleod2010asymptotic} which is a strengthening of a result of Bollob\'as~\cite{Bollobas}.   
	
	In this paper we give a self-contained proof which fully resolves Conjecture~\ref{conj:main} in a strong form.  Recall that for $X,Y$ integer random variables, we define their \textit{total variation distance} by \[\TV(X,Y)=\frac{1}{2}\sum_{k\in \Z} |\Pr(X=k)-\Pr(Y=k)|.\]
	
	\begin{thm}\label{thm:strongmain}
		Let $r=r(n)\ge 2$ be integer valued and divide $2n$, and let $M_n$ be a sequence of perfect matchings in $K_{r\times 2n/r}$.  Let $R_n$ be a uniform random perfect matching on $K_{r\times 2n/r}$ and $X=X(M_n)$ the number of edges shared by~$M_n$ and~$R_n$. Then 
		\[ \TV\lrpar{X,\Po\lrpar{\frac{r}{2r-2}}}\to 0,\]
		as $n$ tends to infinity, where $\Po(\lambda)$ is the Poisson distribution with parameter $\lambda$.
	\end{thm}
	
	Note that $\Pr(X=0)$ is exactly $\f{\PM(K_{r\times 2n/r}-M)}{\PM(K_{r\times 2n/r})}$, so the fact that $\Pr\lrpar{\Po\lrpar{\frac{r}{2r-2}}=0}=e^{-r/(2r-2)}$ gives Conjecture~\ref{conj:main}. Our methods also give the following result, which can be viewed as a generalization of Theorem~\ref{thm:strongmain} when $r$ tends to infinity.
	
	\begin{thm}\label{thm:larger}
		Let $G_n$ be a sequence of graphs such that $G_n$ has $2n$ vertices and has minimum degree $2n-o(n)$, and let $M_n\sub G_n$ be a sequence of perfect matchings.  Let $R_n$ be a uniform random perfect matching on $G_n$ and $X$ the number of edges shared by~$M_n$ and~$R_n$. Then 
     \[ \TV\lrpar{X,\Po\lrpar{\frac{1}{2}}}\to 0, \]
     as $n$ tends to infinity.
	\end{thm}

	\textbf{Notation}.  We typically use $P,Q$ to denote generic perfect matchings, $R$ a uniformly random perfect matching, and $M$ the perfect matching that we are deleting.  If $P$ is a perfect matching and $x$ is a vertex of a graph $G$, we define $P(x)$ to be the unique vertex $y\in V(G)$ such that $xy\in P$.  We write $\N_\ell(G,M)$ to denote the set of perfect matchings of $G$ which contain exactly $\ell$ edges of $M$, and we denote this set simply by $\N_\ell(G,M)$ whenever $G,M$ are clear from context.  Note that $\PM(G)=\sum_{\ell\ge 0} |\N_\ell|$ and $\PM(G-M)=|\N_0|$.  
	
	If $v$ is a vector we let $|v|_1=\sum_i |v_i|$ and $|v|_\infty=\max_i |v_i|$.  Given two functions $f,g$ depending on some parameter $n$, we write $f(n)=O(g(n))$ to mean there exists an absolute constant $C>0$ such that $|f(n)|\le C|g(n)|$ for all $n$.  We emphasize that $C$ is allowed to depend on $r$ whenever we declare $r$ to be a fixed constant.  We also emphasize that expressions like $1+O(n^{-1})$ and $1-O(n^{-1})$ have the same meaning, and we will utilize whichever form seems more natural in a given setting.    We write $f(n)=o(g(n))$ if $f(n)/g(n)$ tends to 0 as $n$ tends to infinity. We write $\alpha=\beta\pm \gamma$ if $\alpha\in [\beta-\gamma,\beta+\gamma]$.

	\section{Proof Sketch}\label{sec:sketch}
	Our proof uses techniques known as ``switching'' arguments, which have been used to obtain precise estimates on the number of various combinatorial objects, see for example \cite{gao2016enumeration,godsil1990asymptotic, mckay2010subgraphs, mckay1991asymptotic}.  A general version of this technique can be found in Hasheminezhad and McKay~\cite{hasheminezhad2010combinatorial}, and a simple example (which is spiritually similar to the argument we pursue in this paper) can be found in an answer of Brendan McKay's to a question of Terence Tao on MathOverflow~\cite{86202}.
	We now sketch the argument we use for proving Conjecture~\ref{conj:main} when $r$ tends towards infinity (which is very similar to the argument used to prove Theorem~\ref{thm:larger}).

	By definition of $\N_\ell=\N_\ell(K_{r\times 2n/r},M)$ being the set of perfect matchings of $K_{r\times 2n/r}$ which contain exactly $\ell$ edges of $M$, we have
	\[\f{\PM(K_{r\times 2n/r})}{\PM(K_{r\times 2n/r}-M)}=\sum_{\ell\ge 0} \f{|\N_\ell|}{|\N_0|}=\sum_{\ell\ge 0} \prod_{k=1}^{\ell} \f{|\N_{k}|}{|\N_{k-1}|}.\]
	With this, we will be done if we can show $\f{|\N_{k}|}{|\N_{k-1}|}\approx \f{1}{2k}$ for all $k$, as this implies that the expression above is approximately
	\[\sum_{\ell\ge 0}\frac{1}{2^\ell\ell!}=e^{1/2},\]
	giving the result.

	To establish bounds on $\f{|\N_{k}|}{|\N_{k-1}|}$, we construct an auxillary bipartite graph $H$ on $\N_{k}\cup \N_{k-1}$ with the property that $\deg_H(P)\approx 2k (2n)^2$ for all $P\in \N_{k}$ and $\deg_H(Q)\approx (2n)^2$ for all $Q\in \N_{k-1}$.  If this holds, then
	\[2k (2n)^2 |\N_{k}|\approx \sum_{P\in \N_{k}}\deg_H(P)=e(H)=\sum_{Q\in \N_{k-1}}\deg_H(Q)\approx (2n)^2|\N_{k-1}|,\]
	and rearranging gives the desired bound.
	
	It remains to construct a bipartite graph $H$ on $\N_{k}\cup \N_{k-1}$ which has the desired degrees.  We do this by making each $P\in \N_{k}$ adjacent to every $Q\in \N_{k-1}$ which can be obtained from $P$ by selecting three edges $ax,by,cz\in P$ and ``rotating'' their endpoints (i.e.\ by deleting these edges and replacing them with $bx,cy,az$), see Figure~\ref{fig:my_label} for an example.  Simple counting arguments will then show that each element of $\N_{k}\cup \N_{k-1}$ has roughly their desired degree.  From this we can prove Conjecture~\ref{conj:main} when $r$ tends towards infinity.
	
	The approach for fixed $r$ is similar but requires some extra difficulties.  In particular, we will not be able to guarantee that every $Q\in \N_{k-1}$ has roughly the degree we are looking for. We overcome this obstacle by studying the behavior of typical perfect matching of $K_{r\times 2n/r}$ (see Proposition~\ref{prop:balancedswitch}) to show that ``most" $Q$ have the correct degrees, and with this we can prove the result.

	\section{Proof of Results}
	We break our proofs into three parts.  First we prove effective estimates on $|\N_k|/|\N_{k-1}|$, where the case of~$r$ fixed for Theorem~\ref{thm:strongmain} is proven under the assumption of a technical result Proposition~\ref{prop:balancedswitch}.  We then show how these ratio estimates imply Theorems~\ref{thm:strongmain} and \ref{thm:larger}.  Finally, we prove Proposition~\ref{prop:balancedswitch}.
	
	\subsection{Ratio Estimates}
	We first consider the case when $G$ has large minimum degree.
	
	\begin{lemma}\label{lem:switchLarge}
		Let $G$ be a $2n$-vertex graph with minimum degree $2n-t$, and let $M\sub G$ be a perfect matching. If $k+t=o(n)$, then
		\[ \frac{|\N_{k}|}{|\N_{k-1}|}=\left(1+o(1)\right)\f{1}{2k}. \]
	\end{lemma}
	\begin{proof}
		For a triple of distinct vertices $(x,y,z)$ and a perfect matching $P$ of $G$, define the $(x,y,z)$-switch of $P$ to be the set of pairs \[P\setminus\{xP(x),yP(y),zP(z)\}\cup \{xP(y),yP(z),zP(x)\},\]  see Figure~\ref{fig:my_label} for an example.  Note that the $(x,y,z)$-switch may not be a subgraph of $G$ if, say, $P(y)$ is not adjacent to $x$.
		\begin{figure}
			\centering
   \begin{tikzpicture}
	\begin{pgfonlayer}{nodelayer}
		\node [style=black] (0) at (-2.5, 1.5) {};
		\node [style=black] (1) at (0, 1.5) {};
		\node [style=black] (2) at (2.5, 1.5) {};
		\node [style=black] (3) at (2.5, 0) {};
		\node [style=black] (4) at (0, 0) {};
		\node [style=black] (5) at (-2.5, 0) {};
		\node [style=none] (6) at (-2.5, 1.75) {$P(x)=Q(z)$};
		\node [style=none] (7) at (0, 1.75) {$P(y)=Q(x)$};
		\node [style=none] (8) at (2.5, 1.75) {$P(z)=Q(y)$};
		\node [style=none] (9) at (-2.5, -0.25) {$x$};
		\node [style=none] (10) at (0, -0.25) {$y$};
		\node [style=none] (11) at (2.5, -0.25) {$z$};
	\end{pgfonlayer}
	\begin{pgfonlayer}{edgelayer}
		\draw (0) to (5);
		\draw (1) to (4);
		\draw (2) to (3);
		\draw[dashed] (5) to (1);
		\draw[dashed] (4) to (2);
		\draw[dashed] (0) to (3);
	\end{pgfonlayer}
\end{tikzpicture}

			\caption{The good $(x,y,z)$-switch of a perfect matching $P$ replaces the three pairs connected by solid lines with the three pairs connected by dashed lines, obtaining the perfect matching $Q$.  This switch is good if all of the lines are edges of $G$ and if $xP(x)$ is the only one of these six edges that is in $M$.}
			\label{fig:my_label}
		\end{figure}
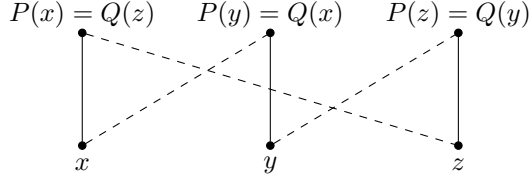
		We say that the triple $(x,y,z)$ is \textit{good} (with respect to $P$ and $M$) if the following hold: 
		\begin{enumerate}
			\item The $(x,y,z)$-switch of $P$ is a subgraph of $G$.  That is, $x\sim P(y),\ y\sim P(z)$, and $z\sim P(x)$.
			\item We have $xP(x)\in M$ and $ yP(y),zP(z),yP(z)\notin M$.
		\end{enumerate}		These conditions imply that if $Q$ is the $(x,y,z)$-switch of some $P\in \N_{k}$ with $(x,y,z)$ good, then $Q\in \N_{k-1}$.  Motivated by this, we define an auxiliary bipartite graph $H$ on $\N_{k}\cup \N_{k-1}$ by making $P\in \N_{k}$ adjacent to every $Q\in \N_{k-1}$ which is the $(x,y,z)$-switch of $P$ for some good triple $(x,y,z)$. 
		\begin{claim}
			Every $P\in \N_{k}$ has $\deg_H(P)=2k(2n-O(k+t))^2$.
		\end{claim}
		\begin{proof}
			Note that $\deg_H(P)$ is exactly the number of good triples $(x,y,z)$ with respect to $P$, so it suffices to count this quantity.  Every good triple can be formed by (i) picking some $x$ which is in an edge of $P\cap M$, (ii) picking some $y\ne x$ which is not a vertex in an edge of $P\cap M$ and such that $P(y)$ is a neighbor of $x$, and (iii) picking some $z\ne x,y$ which is not a vertex in an edge of $P\cap M$ and such that $P(z)$ is a neighbor of $y$, and $z$ is a neighbor of $P(x)$, and $yP(z)\notin M$. We note that these conditions imply that $x,y,z,P(x),P(y),P(z)$ are all distinct vertices.
			
			The number of choices for (i) is exactly $2k$.  The number of choices for each of (ii) and (iii) is certainly at most $2n$ and it is at least $2n-2k-2t-4$ (since, for example, in choosing $z$ one only has to avoid the $2k$ vertices in an edge of $P\cap M$, the $t$ vertices which are not a neighbor of $P(x)$, the $t$ vertices $z$ with $P(z)$ not a neighbor of $y$, and each of the vertices $x,y,M(y)$).  Multiplying the number of choices gives the result.
		\end{proof}
		\begin{claim}
			Every $Q\in \N_{k-1}$ has $\deg_H(Q)=(2n-O(k+t))^2$.
		\end{claim}
		\begin{proof}
			Call a triple of distinct vertices $(x,y,z)$ \textit{reverse good} with respect to $Q$ if $x\sim Q(z),y\sim Q(x),z\sim Q(y)$ and if we have $xQ(z)\in M$ and $yQ(x),zQ(y),yQ(y)\notin M$.  Note that $(x,y,z)$ is reverse good if and only if $Q$ is the $(x,y,z)$-switch of some $P\in \N_k$ with $(x,y,z)$ good (since $Q(x)=P(y),\ Q(y)=P(z),\ Q(z)=P(x)$).  Thus it suffices to count the number of reverse good $(x,y,z)$.
			
			To form such a triple, we (i) pick $x$ which is not in an edge of $Q\cap M$, (ii) let $z=Q(M(x))$, (iii) pick~$y\neq x,z$ which is not in an edge of $Q\cap M$ and such that $Q(y)$ is a neighbour of $z$, and $y$ is a neighbor of $Q(x)$,  and $yQ(x)\notin M$. It is not difficult to see that the output form this procedure will always be reverse good, and that every reverse good triple arises uniquely from this process.  Thus it suffices to bound the number of choices at each step.  For (i) this is $2n-2k$, (ii) is 1 choice, and (iii) is at most $2n$ and at least $2n-2k-2t-4$.  Multiplying these out gives the result.

		\end{proof}
		These two claims and the assumption $k+t=o(n)$ imply
		\[(1+o(1))8kn^2|\N_k|=\sum_{P\in \N_k}\deg_H(P)=e(H)=\sum_{Q\in \N_{k-1}}\deg_H(Q)=(1+o(1))4n^2|\N_{k-1}|,\]
		giving the result.
	\end{proof}
	The case for $G=K_{r\times 2n/r}$ with $r$ fixed is harder and requires the following result, which roughly says that most perfect matchings $P$ of $K_{r\times 2n/r}$ behave ``as expected''.  Here and throughout we let $V_1\cup \cdots\cup V_r$ denote the $r$-partition of $K_{r\times 2n/r}$.

	\begin{prop}\label{prop:balancedswitch}
		Let $r\ge 2$ fixed, $n$ such that $r$ divides $2n$, and $M$ an arbitrary perfect matching on $G=K_{r\times 2n/r}$. Let $1\le k\le\log n$ and let $Q$ be a uniform random perfect matching on $G$ that uses exactly $k-1$ edges of $M$. Then the following holds with high probability: for any $i\in [r]$, the number of vertices $x\in V_i$ such that $Q(M(x))\in V_i$ is $\frac{2n}{r(r-1)}+O(\sqrt{n\log n})$.
	\end{prop}
	
	We defer the proof of Proposition~\ref{prop:balancedswitch} until Subsection~\ref{sec:balanced} and use it to prove the following.
	
	\begin{lemma} \label{lem:switching}
		For any fixed $r\ge 2$ and $k=o(n)$, we have
		\[ \frac{|\N_{k}|}{|\N_{k-1}|}=(1+o(1)) \frac{r}{(2r-2)k} \]
	\end{lemma}
	\begin{proof}
		As before, given a triple of distinct vertices $(x,y,z)$, we define the $(x,y,z)$-switch of $P$ to be \[P\setminus\{xP(x),yP(y),zP(z)\}\cup \{xP(y),yP(z),zP(x)\}.\]  Note that in general the $(x,y,z)$-switch of $P$ might not be a perfect matching.  For $P\in \N_{k}$, we say that a triple $(x,y,z)$ is \textit{good} (with respect to $P$ and $M$) if the following hold:
		\begin{enumerate}
			\item We have $x,y,z\in V_i$ for some $i$.
			\item We have $xP(x)\in M$ and $yP(y),zP(z),yP(z)\notin M$.
		\end{enumerate}
		These conditions ensure that if $Q$ is an $(x,y,z)$-switch of $P\in \N_{k}$ with $(x,y,z)$ good, then $Q\in \N_{k-1}$.  We define the bipartite graph $H$ on $\N_k\cup \N_{k-1}$ by having each $P\in \N_{k}$ adjacent to every $Q\in \N_{k-1}$ which is an $(x,y,z)$-switch of $P$ for some good $(x,y,z)$.
		
		\begin{claim}
			For every $P\in \N_{k}$, we have
			\[\deg_H(P)=2k\left(\frac{2n}{r}-O(k)\right)^2\]
		\end{claim}
		\begin{proof}
			Note that $\deg_H(P)$ is exactly the number of triples $(x,y,z)$ which are good with respect to $P$.  To form such a triple, we first pick any $x$ such that $xP(x)\in M$, the number of choices of which is exactly $2k$ by definition of $\N_{k}$.  If one has chosen $x\in V_i$, then the number of choices for $y\in V_i\setminus\{x\}$ with $yP(y)\notin M$ is at most $|V_i|=2n/r$ and at least $2n/r-1-k$.  Given such $x,y$, the number of choices for $z\in V_i\setminus\{x,y\}$ with $zP(z)\notin M$ and $yP(z)\notin M$ is between $2n/r$ and $2n/r-3-k$.  As each good triple can be formed uniquely by this procedure, we conclude the result.
		\end{proof}
		Motivated by Proposition~\ref{prop:balancedswitch}, we define $\F_{k-1}\subseteq \N_{k-1}$ to be the set of perfect matchings $Q$ such that for any $i\in [r]$, the number of vertices $x\in V_i$ with $Q(M(x))\in V_i$ is $\frac{2n}{r(r-1)}+O(\sqrt{n\log n})$.
		\begin{claim}
			For every $Q\in \F_{k-1}$, we have
			\[\deg_H(Q)=\left(\frac{2n}{r-1}+O\left(\sqrt{n\log n}+k\right)\right)\left(\frac{2n}{r}-O(k)\right),\]
			and for every $Q\in \N_{k-1}\setminus \F_{k-1}$, we have
			\[\deg_H(Q)\le 4n^2.\]
		\end{claim}
		\begin{proof}
			We say that a triple of distinct vertices $(x,y,z)$ is \textit{reverse good} for $Q$ if $x,y,z\in V_i$ for some $i$ and if $xQ(z)\in M$ and $yQ(x),zQ(y),yQ(y),zQ(z)\notin M$.  Note that $Q$ is a $(x,y,z)$-switch of some $P\in \N_{k}$ with $(x,y,z)$ good if and only if $(x,y,z)$ is reverse good, so it suffices to count how many triples have this property.
			
			Observe that for any reverse good triple, we have that $z=Q(M(x))$ is in the same part as $x$ (since $M(x)=Q(z)$).  Thus any such triple can be formed (uniquely) by first picking $x$ such that $Q(M(x))\ne x$ lies in the same part $V_i$ as $x$, setting $z=Q(M(x))$, and then choosing any $y\in V_i\setminus\{x,z\}$ such that $yQ(x),zQ(y),yQ(y)\notin M$.
			
			There are trivially at most $4n^2$ ways to carry out this procedure, so we conclude $\deg_G(Q)\le 4n^2$ for all $Q\in \N_{k-1}$, and we may assume $Q\in \F_{k-1}$ from now on.  By definition of $\F_{k-1}$, the number of choices for $x$ is at most $\frac{2n}{r-1}+O(\sqrt{n\log n})$ and at least $\frac{2n}{r-1}+O(\sqrt{n\log n})-2k$ (since $Q(M(x))=x$ is only possible if $xQ(x)\in M$).  Given this, the number of choices for $y$ is at most $2n/r$ and at least $2n/r-4-k$.  Putting this together gives the claim.
		\end{proof}
		This last claim and $k=o(n)$ implies
		\[(1+o(1))\frac{4n^2}{r(r-1)}|\F_{k-1}|\le \sum_{Q\in \N_{k-1}}\deg_H(Q)\le (1+o(1))\frac{4n^2}{r(r-1)}|\F_{k-1}|+4n^2|\N_{k-1}\setminus\F_{k-1}|.\]
		Since $|\F_{k-1}|=(1+o(1))|\N_{k-1}|$ by Proposition~\ref{prop:balancedswitch}, the above gives
		\[\sum_{Q\in \N_{k-1}}\deg_H(Q)=(1+o(1))\frac{4n^2}{r(r-1)}|\N_{k-1}|.\]
		Using this and the first claim gives
		\[(1+o(1))\frac{8kn^2}{r^2}|\N_{k}|=\sum_{P\in \N_{k}} \deg_H(P)=\sum_{Q\in \N_{k-1}}\deg_H(Q)=(1+o(1))\frac{4n^2}{r(r-1)}|\N_{k-1}|,\]
		and we conclude the result.
	\end{proof}
	\subsection{Using Ratio Estimates}
	
	It remains to use the ratio bounds to establish our results, and for this we use the following. 
	\begin{prop}\label{prop:main}
		Let $G_n$ be a sequence of graphs with $|V(G_n)|$ tending towards infinity such that each $G_n$ has an even number of vertices, $\Theta(|V(G_n)|^2)$ edges, and contains some perfect matching $M_n\sub G_n$.  Let $R_n$ be a uniform random perfect matching on $G_n$ and $X=X(M_n)$ the number of edges shared by~$M_n$ and~$R_n$.  If there exists a fixed real number $\alpha>0$ such that for all $k=o(|V(G_n)|)$ we have
		\[\frac{|\N_k|}{|\N_{k-1}|}=(1+o(1))\frac{\alpha}{k},\]
		and such that for all $e\in G_n$ we have
		\[\Pr(e\in R_n)=O\left(\frac{1}{|V(G_n)|}\right),\]
		then
		\[ \TV\lrpar{X,\Po\lrpar{\alpha}}\to 0. \]
	\end{prop}
	\begin{proof}
		By assumption, there exists a function $\omega:=\omega(n)$ growing arbitrarily slowly to infinity such that for any $\ell\le \omega$,     \begin{equation} \frac{|\N_\ell|}{|\N_0|} = \prod_{1\le k\le \ell} \frac{|\N_{k}|}{|\N_{k-1}|}=\lrpar{1+o\lrpar{\frac{1}{\omega}}} \frac{\alpha^\ell}{\ell!}.\label{eq:omega}\end{equation}
		Also by assumption and linearity of expectation, we have
		\[ \E X= \frac{|V(G_n)|}{2}\cdot O\left(\frac{1}{|V(G_n)|}\right)= O(1).\]
		By Markov's inequality we find 
		\begin{equation} \frac{\sum_{k\ge \omega}|\N_k|}{\sum_{k\ge 0} |\N_k|}=\Pr(X\ge \omega)=O(1/\omega).\label{eq:Markov}\end{equation}
		Using \eqref{eq:Markov} and \eqref{eq:omega} gives
		\[\sum_{k\ge 0} |\N_k|=(1+O(1/\omega))\sum_{k=0}^\omega |\N_k|=(1+O(1/\omega)) |\N_0|\sum_{k=0}^\omega \frac{\alpha^k}{k!}=(1+O(1/\omega)) e^{\alpha}|\N_0|, \]
		where we use $e^{\alpha}=\sum_{k\ge 0}\frac{\alpha^k}{k!}$ and the routine estimate \begin{equation}\sum_{k>\omega} \frac{\alpha^k}{k!}=O\left(\frac{\alpha^{\omega+1}}{(\omega+1)!}\right).\label{eq:Poisson}\end{equation}  This together with \eqref{eq:omega} gives for all $\ell\le \sqrt{\omega}$ that
		
		\begin{equation} \Pr(X=\ell)=\frac{|\N_\ell|}{\sum_{k\ge 0} |\N_k|}= (1+O(1/\omega))\frac{|\N_\ell|}{e^{\alpha}|\N_0|}=(1+O(1/\omega)) \frac{\alpha^\ell}{\ell!}e^{-\alpha}. \label{eq:whatever} \end{equation}
		Therefore,
		\begin{align*}
			2\TV(X,\Po(\alpha))&= \sum_{k\ge 0} \abs{\Pr(X=k) - \frac{1}{k!} \alpha^k e^{-\alpha}  }\\
			&\le \sum_{0\le k\le \sqrt{\omega}} \abs{\Pr(X=k) - \frac{1}{k!} \alpha^ke^{-\alpha}  } + \Pr(X>\sqrt{\omega})+\Pr(\Po(\alpha)
			>\sqrt{\omega})\le O(1/\sqrt{\omega}),
		\end{align*}
		where this last step used \eqref{eq:whatever}, \eqref{eq:Markov}, and \eqref{eq:Poisson}.
	\end{proof}
	
	We note that the condition $\Pr(e\in R_n)=O(|V(G_n)|^{-1})$ in Proposition~\ref{prop:main} will automatically apply whenever the $G_n$ are edge-transitive, and in particular when $G_n=K_{r\times 2n/r}$.  With this in mind, we prove our main result.
	
	\begin{proof}[Proof of Theorem~\ref{thm:strongmain}]
		Let $\eps\in (0,1)$ be an arbitrary fixed constant. We will show that there exists a constant $n_\eps$ such that for all $n\ge n_\eps$, $r|2n$, and perfect matching $M_n$ of $K_{r\times 2n/r}$, the random variable $X=X(M_n)$ satisfies  
		\[ d(n,r):=\TV\lrpar{X,\Po\lrpar{\frac{r}{2r-2}}}\le \eps.  \]
		Let $t(\eps,n)$ be the largest number such that for all $r\le t(\eps,n)$ with $r|2n$, we have 
		$d(n,r)\le \eps$.  
		Lemma~\ref{lem:switching} and Proposition~\ref{prop:main} implies that for any fixed integer $c$, we have $t(\eps,n)\ge c$ for all large enough $n$. Therefore $t(\eps,n)\to\infty$ as $n\to\infty$. 
		Now from Lemma~\ref{lem:switchLarge} and Proposition~\ref{prop:main}, it follows that for all large enough $n$ and $r\ge t(\eps,n)$, $r|2n$, we have 
		\[ d(n,r)\le \TV\lrpar{X,\Po\lrpar{\frac{1}{2}}}+\TV\lrpar{\Po\lrpar{\frac{1}{2}},\Po\lrpar{\frac{r}{2r-2}}}\le \eps,\]
		which completes the proof.
	\end{proof}    
	Finally, we prove Theorem~\ref{thm:larger} when working with graphs of minimum degree $2n-o(n)$.
	
	\begin{proof}[Proof of Theorem~\ref{thm:larger}]
	    By Proposition~\ref{prop:main} and Lemma~\ref{lem:switchLarge}, it suffices to show that if $G$ is a $2n$-vertex graph with minimum degree $2n-o(n)$, then a uniform random perfect matching $R$ of $G$ satisfies $\Pr(e\in R)=O(1/n)$ for all $e\in E(G)$.  We accomplish this with a switching argument.
	    
	    Fix an edge $e=xy$ in $G$.  Let $\N_e$ be the set of perfect matching on $G$ that contains $e$, and let $\N_e^c$ be the set of perfect matching not containing $e$. We consider an auxiliary graph $H$ on $\N_e\cup \N_e^c$ where $P\in \N_e$ and $Q\in \N_e^c$ are adjacent if there exists an edge $uv\in P\setminus\{xy\}$ such that $Q=P\setminus\{uv,xy\}\cup \{uy,xv\}$.  
	    
	    It is not difficult to see that $\deg_H(Q)\le 1$ for all $Q\in \N_e^c$.  Note that $P$ has a neighbor in $H$ for each edge $uv\in P\setminus\{xy\}$ with $u\in N_G(y)$ and $v\in N_G(x)$, and the number of such edges is at least $n-o(n)$ by the minimum degree condition.  With this
    \[  (1+o(1))n|\N_e|=\sum_{P\in \N_e} \deg_H(P)=\sum_{Q\in \N_e^c} \deg_H(Q) \le  |\N_e^c|.\]
    Therefore 
    \[ \Pr(e\in R) \le \frac{|\N_e|}{|\N_e|+|\N_e^c|}=O\lrpar{\frac{1}{n}} .\]
\end{proof}

	\subsection{Proof of Proposition~\ref{prop:balancedswitch}}\label{sec:balanced}

	We prove Proposition~\ref{prop:balancedswitch} through the following technical result.

	\begin{prop}\label{prop:balanced}
		Let $r\ge 3$ be fixed, and let $v_1,\dots, v_r$ be integers such that $v_i=n-O(\log n)$ and $\sum_{i\in [r]} v_i$ is even. Let $G$ be a complete $r$-partite graph with partition $V_1,\dots, V_r$ where $|V_i|=v_i$ for all $i\in [r]$. Let $R$ be a uniform random perfect matching of $G$. Then with high probability, the number of edges of $R$ between $V_i$ and $V_j$ equals $\frac{n}{r-1}+O(\sqrt{n\log n})$ for all $i,j$.
	\end{prop}
	
	We emphasize that the $n$ of Proposition~\ref{prop:balanced} should be thought of as roughly $2/r$ times the $n$ of Proposition~\ref{prop:balancedswitch}.  Throughout the proof we assume $n$ is sufficiently large in terms of $r$, as the result is trivial otherwise. Consider vectors
	\[v=(v_{i,j})_{i<j,\ i,j\in [r]}\in \Z^{\binom{r}{2}}.\]
	Note that there exists a perfect matching $R$ in $G$ with $v_{i,j}$ edges between parts $V_i$ and $V_j$ for all $i<j$ if and only if the vector $v$ satisfies the following system of constraints (PM):
	\begin{align}
		v_{i,j}&\ge 0 \ \forall i,j\\
		\sum_{j\in [r]\setminus \{i\}} v_{i,j}&=v_i\ \forall i, \text{ where  $v_{i,j}:=v_{j,i}$ if $i>j$}.
	\end{align}
	Let 
	\[\cp=\{v\in \Z^{\binom{r}{2}} | v \text{ satisfies (PM)}\},\]
	and let $c=\floor{n/(r-1)}, d=\floor{\sqrt{n\log n}}$.  Note that $c$ is essentially the average number of edges we expect between two parts in a random matching.  
	
	The rough strategy for the rest of the proof is as follows.  We partition the space $\Z^{r\choose 2}$ into cubes $\B_u$ of side length $2d$ and let $\cp_u=\B_u\cap \cp$ denote the set of vectors in $\B_u$ that satisfy (PM).  We let $\C$ denote the union of the $\cp_u$ subcubes with $v_{i,j}=c\pm 3d$ for all $v\in \cp_u$.  With this, Proposition~\ref{prop:balanced} essentially says that with high probability, a uniform random $R$ corresponds to some vector in $\C$ with high probability.  To prove this, we show  (i) that there are not too many vectors of $\cp$ outside of $\C$ and (ii) that the probability of a uniformly random perfect matching $R$ corresponding to any given $v\notin \C$ is very small.  With this we can conclude our result.  We now move on to the precise details.
	
	For any integral vector $u=(u_{i,j})_{i<j\in[r]}$ (possibly with negative entries), let 
	\begin{align*}
		\B_u&:=\prod_{i<j,\ i,j\in[r]} (c+2u_{i,j}d-d, c+2u_{i,j}d+d] \subseteq \Z^{\binom{r}{2}},\\
		\cp_u&:= \cp\cap \B_u,\\
		\C&:=\bigcup_{u:|u|_\infty\le 1} \cp_u.
	\end{align*}
	We start with a small observation.
	\begin{lemma}\label{lem:nonempty}
		The set $\cp_0$ is non-empty.
	\end{lemma}
	\begin{proof}
		We will construct a perfect matching which corresponds to a vector in $\cp_0$ as follows. Arbitrarily order the pairs $(i,j), 1\le i<j\le r$, then sequentially put an arbitrary matching with $n/(r-1)-\sqrt{n\log n}/r$ edges between $V_i$ and $V_j$ such that the current matching avoids the vertices used in all previous matchings. Let $P'$ be the matching at the end of this process.  Observe that the subgraph of $G$ induced by the vertices which are not in $P'$ is a complete $r$-partite graph where each partition has $\frac{r-1}{r}\sqrt{n\log n}+O(\log n)$ vertices. By Dirac's theorem and the fact that $r\ge 3$, there exists a perfect matching $P$ that completes $P'$. It is clear that for any $i<j$, the number of edges between any two parts lies in 
		\[\left[ \frac{n}{r-1}-\frac{\sqrt{n\log n}}{r}, \frac{n}{r-1} +\frac{r-2}{r}\sqrt{n\log n}+O(\log n)\right], \]
		proving the result.
	\end{proof}
	This allows us to prove that there are not too many elements in any given set $\cp_u$.
	\begin{lemma}\label{lem:densecenter}
		For any integral vector $u=(u_{i,j})_{i<j\in[r]}$, we have $|\cp_u|\le |\C|$.
	\end{lemma}
	
	\begin{proof}
		By Lemma~\ref{lem:nonempty} there exits some $\vc\in \cp_0$.  We claim that for all $v,v'\in \cp_u$, we have $\vc+(v-v')\in \C$.
		
		Indeed, since $v,v'\in \cp_u\subseteq \B_u$, we have $|v-v'|_\infty< 2d$, and hence 
		\[ |\vc+v-v'-c|_\infty\le |\vc-c|_\infty+|v-v'|_\infty< 3d.\]
		To complete the claim, it remains to show that $\vc+v-v'$ satisfies (PM). We have $(\vc+v-v')_{i,j}\ge 0$ by the inequality above and $c\ge 3d$ for $n$ sufficiently large in terms of $r$. Since $v,v'\in \cp$, for any $i$ we have 
		\[ \sum_{j\in [r]\setminus \{i\}} v_{i,j}=\sum_{j\in [r]\setminus \{i\}} v'_{i,j}=v_i,\]
		so $\vc\in \cp$ implies
		\[ \sum_{j\in [r]\setminus \{i\}} (\vc+v-v')_{i,j}=\sum_{j\in [r]\setminus \{i\}} \vc_{i,j}=v_i,\]
		proving the claim.
		
		If $\cp_u=\emptyset$ then there is nothing to prove.  Otherwise, fix some arbitrary $v'\in \cp_u$ and define the map $f(v)=\vc+(v-v')$.  Note that $f$ is injective and maps $\cp_u$ to $\C$ by our claim, giving the result.
	\end{proof}

	We will need the following technical lemma.
	\begin{lemma}\label{lem:optimise}
		Let $x_1\ge \ldots \ge x_t\ge 0$ and $y_1\ge \ldots, y_t\ge 0$ be non-negative integers such that $\sum_{i=1}^t x_i=\sum_{i=1}^t y_i:=S$. Let $c=\floor{S/t}$ and suppose that $|x_i-c|\le k$ for all $i\in [t]$ and $y_1\ge c+2\delta$. Then
		\[ \prod_{i=1}^t \frac{x_i!}{y_i!}\le \exp\left( -\frac{\delta(\delta-1)}{c+\delta} + \frac{2k^2t}{c-k} \right). \]
	\end{lemma}
	\begin{proof}
		In order to estimate $\prod_{i=1}^t z_i!$ for any $z_1\ge\ldots\ge z_t\ge 0$ with $\sum_{i=1}^t z_i=S$, we will iteratively construct sequences of integers $\gam_1^{(\ell)}\ge \cdots \ge \gam_t^{(\ell)}$ with $\sum_{i=1}^t \gam_i^{(\ell)}=S$ such that $\gam_i^{(0)}$ is as close to $c$ as possible for all $i$, and such that  $\gam_i^{(\ell)}=z_i$ for the final value of $\ell$.  By keeping track of how much $\prod_{i=1}^t \gam_i^{(\ell)}!$ changes at each step, we will be able to effectively estimate $\prod_{i=1}^t z_i!$.
		
		Fix an integer sequence $z_1\ge\ldots\ge z_t\ge 0$ with $\sum_{i=1}^t z_i=S$, and let $c+1\ge \gamma_1^{(0)}\ge \dots \gamma_t^{(0)}\ge c$ be integers such that $\sum_{i=1}^t \gamma_i^{(0)}=S$. Observe that $\gam_i^{(0)}$ has the property that there is a $j\in [t]$ such that $z_i\ge \gamma_i^{(0)}$ for all $i\le j$ and $z_i< \gamma_i^{(0)}$ for all $i>j$. Iteratively given $\gam_i^{(\ell-1)}$ which continues to satisfy this property,   we pick the smallest index $i_1$ such that $\gamma_{i_1}^{(\ell-1)}<z_{i_1}$ and the largest index $i_2$ such that $\gamma_{i_2}^{(\ell-1)}>z_{i_2}$. We then set
		$\gamma_{i_1}^{(\ell)}=\gamma_{i_1}^{(\ell-1)}+1$, $\gamma_{i_2}^{(\ell)}=\gamma_{i_2}^{(\ell-1)}-1$, and keep $\gam_i^{(\ell)}=\gam_i^{(\ell-1)}$ for all other $i$; noting that with this, $\gam_i^{(\ell)}$ satisfies the desired property and $\sum_{i=1}^t \gam_i^{(\ell)}=S$.  We terminate this process when $\gamma_i^{(\ell)}=z_i$ for all $i$, and we let $n(z_1,\dots, z_t)$ be the step at which we terminate. Define 
		\[ \Gamma(z_1,\dots, z_t)^{(\ell)}:=\Gamma^{(\ell)}= \prod_{i=1}^t \gamma_i^{(\ell)}!, \]
		and we make the simple observations that $\Gamma^{(\ell)}\ge \Gamma^{(\ell-1)}$ and $\Gamma^{(n(z_1,\ldots,z_t))}=\prod_{i=1}^t z_i!$. 
		
		By using this approach for $x_1,\dots, x_t$, observing that $n(x_1,\dots, x_t)\le kt$ and $\Gamma^{(\ell)}/\Gamma^{(\ell-1)}\le (c+k)/(c-k)$ for all $\ell$, we have 
		\begin{equation}\label{eq:opt1}
			\frac{\prod_{i=1}^t x_i!}{\Gamma^{(0)}}\le \lrpar{\frac{c+k}{c-k}}^{kt}\le \exp\lrpar{ \frac{2k^2t}{c-k} }.
		\end{equation}
		
		Similarly, by using this approach for $y_1,\dots, y_t$ and observing that for $\ell=\delta,\delta+1,\dots, 2\delta-1$ we have $\Gamma^{(\ell)}/\Gamma^{(\ell-1)}\ge (c+\delta)/(c+1)$ (since we will increase the value of $\gamma_1^{(\ell-1)}\ge c+\delta-1$ by 1 and decrease the value of $\gamma_{i_2}^{(\ell-1)}\le c+1$ by 1 for some $i_2$), and hence
		\begin{align}\label{eq:opt2}
			\frac{\Gamma^{(0)}}{\prod_{i=1}^t y_i!}\le \prod_{\delta\le \ell< 2\delta} \frac{\Gamma^{(\ell-1)}}{\Gamma^{(\ell)}} \le  \lrpar{\frac{c+1}{c+\delta}}^\delta\le \exp\lrpar{-\frac{\delta(\delta-1)}{c+\delta}}.
		\end{align}
		Combining \eqref{eq:opt1} and \eqref{eq:opt2} gives the desired result. 
		
	\end{proof}
	For any $v=(v_{i,j})_{i<j\in [r]}\in \cp$ and $\mathcal{S}\subseteq \cp$, let 
	\begin{align*}
		\phi(v)&:=\Pr(\text{uniform random perfect matching $R$ on $G$ has }v_{i,j}\text{ edges between }V_i,V_j\ \forall i<j), \\
		\phi(\mathcal{S})&:=\sum_{v\in \mathcal{S}} \phi(v).
	\end{align*}
	Note that
	\begin{equation} \phi(v)\propto \prod_{i\in [r]} \binom{v_i}{v_{i,1},\dots, v_{i,r}} \prod_{i<j} v_{i,j}!\propto \prod_{i<j} \frac{1}{v_{i,j}!}.\label{eq:weight}\end{equation}
	\begin{lemma}\label{lem:singleratio}
		For any $u\in \Z^{\binom{r}{2}}, |u|_1\ge 100r^4$, $v\in \cp_u, v'\in \C$, we have
		\[ \frac{\phi(v)}{\phi(v')}\le n^{-50r^8}. \]
	\end{lemma}
	One can easily improve the bound of this lemma with a more careful argument, but we do not make any attempt at optimizing our bound beyond what is needed.
	\begin{proof}
		First we note that
		\begin{align*}
			\sum_{i<j} v_{i,j}&=\sum_{i<j} v_{i,j}'=\frac{1}{2}\sum_{i} v_i=:S.
		\end{align*}
		Since $S=\frac{rn}{2}+O(\log n)$, we have
		\begin{align*}	\abs{v'_{i,j}-S/\binom{r}{2}}&\le 3d+O(\log n)=:k \quad \forall\ i<j,
		\end{align*}
		and that
		\begin{align*}
			\frac{nr}{2}+O(\log n)=\sum_{i<j} v_{i,j}&\le \frac{nr}{2}+2d\sum_{i<j} u_{i,j}+d\frac{r(r-1)}{2}.
		\end{align*}
		This implies \[\sum_{i<j} u_{i,j}\ge-\frac{r(r-1)}{4}+o(1) .\]
		Using this together with $|u|_1\ge 100r^4$ gives 
		\[ \sum_{i<j, u_{i,j}\ge 0} u_{i,j}\ge 50r^4- \frac{r(r-1)}{8}+o(1)\ge 49r^4.\]
		Therefore, $\max_{i,j} u_{i,j}\ge 49r^4$, and hence a crude bound gives
		\[ \max_{i<j} v_{i,j}\ge c+2 \max_{i<j} u_{i,j}d-d\ge  S/\binom{r}{2}+96r^4d=:S/\binom{r}{2}+2\delta.\]
		Therefore by \eqref{eq:weight} and Lemma~\ref{lem:optimise} and some crude estimates, 
		\begin{align*}
			\frac{\phi(v)}{\phi(v')}= \prod_{i<j} \frac{v_{i,j}'!}{v_{i,j}!}
			&\le\exp\left(-\frac{\delta(\delta-1)}{c+\delta} + \frac{2k^2\binom{r}{2}}{c-k}\right)\\
			&\le  \exp( -100r^8\log n + 18  r^3 \log n )\\
			&\le \exp(-50 r^8 \log n)
		\end{align*}
		which gives the desired result.
	\end{proof}
	We now have all we need to prove Proposition~\ref{prop:balanced}.
	
	\begin{proof}[Proof of Proposition~\ref{prop:balanced}]
		By Lemmas~\ref{lem:densecenter} and \ref{lem:singleratio}, for any $u$ with $|u|_1\ge 100r^4$,
		\[ \frac{\phi(\cp_u)}{\phi(\C)}\le \f{\sum_{v\in \cp_u}\phi(v)}{|\cp_u|\min_{v'\in \cp_u}\phi(v')} \le n^{-50r^8}.\]
		Since the number of $u$ such that $\cp_u$ is non-empty is (very crudely) at most $n^{\binom{r}{2}}$, the inequality above gives
		\[ \sum_{u: |u|_1\ge 100r^4} \frac{\phi(\cp_u)}{\phi(\C)}\le n^{-r^8}, \]
		which concludes the proof.
	\end{proof}
	Finally, we prove Proposition~\ref{prop:balancedswitch}, which we restate for convenience.
	\begin{propn}[\ref*{prop:balancedswitch}]
		Let $r\ge 2$ fixed, $n$ such that $r$ divides $2n$, and $M$ an arbitrary perfect matching on $G=K_{r\times 2n/r}$. Let $1\le k\le\log n$ and let $Q$ be a uniform random perfect matching on $G$ that uses exactly $k-1$ edges of $M$. Then the following holds with high probability: for any $i\in [r]$ the number of vertices $x\in V_i$ such that $Q(M(x))\in V_i$ is $\frac{2n}{r(r-1)}+O(\sqrt{n\log n})$.
	\end{propn}
	\begin{proof}
		The result is trivial if $r=2$, so we assume $r\ge 3$.  We prove the following stronger statement: let $M^*$ be a subset of $M$ with $|M^*|=k-1\le \log n$, and  
		let $Q$ be a uniform random perfect matching on $G$ that intersects $M$ at exactly $M^*$.  Then the conclusion of Proposition~\ref{prop:balancedswitch} still holds.
		
		Let $G'$ be the graph obtained by removing the endpoints of $M^*$ in $G$, leaving the complete $r$-partite graph with vertex set $V_1',\dots, V_r'$. Let $Q'$ be the restriction of $Q$ to $G'$, let $M'$ be the restriction of $M$ to $G$, and let $R$ be a uniform random perfect matching on $G'$.  It is not too difficult to see that $Q'$ is a uniform random perfect matching on $G'$ that avoids $M'$.
		
		For any matching $P$ on $G'$, we define $P_{i,j}\subseteq V_i'$ to be the set of vertices $x\in V_i'$ such that $P(x)\in V_j'$. We note that the graph $G'$ satisfies the conditions of Proposition~\ref{prop:balanced} (with $n$ replaced by $2n/r$). Therefore, for any $i< j\in [r]$, with high probability $|R_{i,j}|=\frac{2n}{r(r-1)}+O(\sqrt{n\log n})$. By symmetry, we know that if we condition on $|R_{i,j}|=m$, then $R_{i,j}$ is the uniform random subset of $V'_i$ with size $m$. Combining these two facts, from standard concentration inequalities (see~\mbox{\cite[Theorem~2.1 and~2.10]{JLR}}) we get that with high probability
		\begin{equation} |M'_{i,j}\cap R_{i,j}|= \frac{|M'_{i,j}|}{r-1} +O(\sqrt{n\log n}). \label{eq:MT} \end{equation}
		
		Given a perfect matching $P$ of $G'$, let 
		\[ S(P)=(P_{i,j})_{i\neq j \in [r]},\]
		and let $\mathcal{S}=\{S(P):P\subseteq G'-M'\}$.
		\begin{claim}
			There exists a constant $C_r$ such that for all $S\in \mathcal{S}$, we have 
			\[ \frac{ \Pr(S(Q')=S) }{\Pr(S(R)=S)}\le C_r. \]
		\end{claim}
		\begin{proof}
			Recall that $d_N$ is the number of derangements of order $N$, and define 
			\[\alpha=\sup_{N\ge 2} \frac{N!}{d_N}.\]
			Note that $\alpha<\infty$ since we restrict to $N\ge 2$ and since $\lim_{n\to \infty} \frac{d_N}{N!}=e^{-1}$ by \eqref{eq:deranged}, and that $\alpha\ge 1$.  We will prove our result with $C_r=2\alpha^{r\choose 2}$.
			
			To this end, for $S\in \mathcal{S}$ we let $\PM(G',S)$ denote the number of perfect matchings $U$ of $G'$ with $S(U)=S$, and we similarly define $\PM(G'-M',S)$.  We claim that 
			\[\PM(G',S)\le \alpha^{r\choose 2}\PM(G'-M',S).\]
			Indeed, we have $\PM(G',S)=\prod_{i<j} |S_{i,j}|!$.  Because $S\in \mathcal{S}$, there exists some perfect matching $P\subseteq G'-M'$ with $S(P)=S$.  In particular, for all $i\ne j$ we either have $|S_{i,j}|\ge 2$, or we have $S_{i,j}=\{x\},S_{j,i}=\{y\}$ with $xy\notin M'$ (as otherwise no such $P\subseteq G'-M'$ could exist).  Letting $d_{i,j}=1$ if this latter case happens and otherwise setting $d_{i,j}=d_{|S_{i,j}|}$, we see that \[\PM(G'-M',S)\ge \prod_{i<j} d_{i,j},\]
			as there are at least $d_{i,j}$ ways to choose the matching between $S_{i,j}$ and $S_{j,i}$ while avoiding $M$.  By definition we have $|S_{i,j}|!/d_{i,j}\le \alpha$ for all $i,j$, and from this the subclaim follows.

            For any perfect matching $P$ of $G'$, $S(P)\notin \mathcal{S}$ only if $P_{i,j}=1$ for some $i,j$. Therefore by Proposition~\ref{prop:balanced}, $S(R)\in \mathcal{S}$ with high probability.
			To finish the proof, we observe that
			\begin{align*}
			    \Pr(S(Q')=S)=\frac{\PM(G'-M',S)}{\sum_{S'\in \mathcal{S}}\PM(G'-M',S')}\le \frac{\PM(G',S)}{\alpha^{-\binom{r}{2}}\sum_{S'\in \mathcal{S}}\PM(G',S')} &=\alpha^{\binom{r}{2}} \Pr(S(R)=S|S(R)\in \mathcal{S}) \\ 
       &\le C_r\Pr(S(R)=S).
			\end{align*}
		\end{proof}
		With this claim and \eqref{eq:MT}, we have with high probability
		\[ |M'_{i,j}\cap Q'_{i,j}|= \frac{|M'_{i,j}|}{r-1} +O(\sqrt{n\log n}).\]
		This gives the desired result, as the number of vertices $x\in V_i$ such that $Q(M(x))\in V_i$ is (deterministically)
		\[ \sum_{j\neq i} |M'_{j,i}\cap Q'_{j,i}| \pm (k-1),\]
		which is $2n/(r(r-1))+O(\sqrt{n\log n})$ with high probability since $\sum_{j\ne i}|M'_{j,i}|=|V'_i|=2n/r-O(\log n)$.
	\end{proof}

	\section{Further Directions}\label{sec:conclusion}
	In this paper we proved Conjecture~\ref{conj:main} by utilizing switching arguments, and there are a number of extensions one could consider.  One such direction is to try and estimate $\PM(K_{r\times 2n/r}-D)/\PM(K_{r\times 2n/r})$ when $D$ is a $d$-regular subgraph of $K_{r\times 2n/r}$ for some $d>1$.  It is possible that more complicated versions of our arguments here could be effective for this problem.
	
	A different direction in the spirit of Theorem~\ref{thm:larger} is to see to what extent Conjecture~\ref{conj:main} can be generalized to graphs $G$ other than $K_{r\times 2n/r}$.  For example, is it true that for all $\alpha>1$, a sequence of $2n$-vertex $\alpha n$-regular graphs $G$ satisfies \[\lim_{n\to \infty} \frac{\PM(G-M)}{\PM(G)}=e^{-1/\alpha}\]
	for any perfect matching $M\subseteq G$?  It seems likely that this statement is far too strong to be true, but we do not know of any counterexamples.  We note that the question for $\alpha<1$ does not make sense since $G$ may not have a perfect matching, and that the result is false at $\alpha=1$.  In particular, the result fails at $\alpha=1$ by considering $G$ to be $K_n\cup K_n$ together with a perfect matching $M$ (since for $n$ odd, $\PM(G-M)=0$).

	\textbf{Acknowledgements}.  We thank Dan Johnston for informing us of an updated version of \cite{johnston2022deranged}, Mark Kayll for suggesting edits to the original manuscript, and Cory Palmer for telling us about Conjecture~\ref{conj:main} at the BSM 100/3 reunion conference.
	\bibliographystyle{abbrv}
	\bibliography{bibliography}
\end{document}